\documentclass{amsart}

\newtheorem{theorem}{Theorem}[section]
\newtheorem{corollary}[theorem]{Corollary}
\newtheorem{lemma}[theorem]{Lemma}
\newtheorem{proposition}[theorem]{Proposition}
 \theoremstyle{definition}
 \newtheorem{definition}[theorem]{Definition}
 \theoremstyle{remark}
 \newtheorem{remark}[theorem]{Remark}
 \newtheorem*{example}{Example}
 \numberwithin{equation}{section}

\begin{document}

\title[Regularity of differential operators of constant strength]
{Smoothness of solutions of differential equations of constant
strength in Roumieu spaces}
\author[R. CHAILI]{Rachid CHAILI}
\address{Department of Mathematics \\
University of Sciences and Technology of Oran\\
Algeria}
\email{rachidchaili@gmail.com}

\author[T. MAHROUZ]{Tayeb MAHROUZ}
\address{Department of Mathematics \\
University of Ibn Khaldoun Tiaret\\
Algeria}
\email{mahrouz78@gmail.com}

\thanks{The authors are supported by Laboratory of mathematical analysis and
applications, Univ-Oran1.}



\date{}

\begin{abstract}
We show in this work that every solution of hypoelliptic differential
equations with constant strength with coefficients in Roumieu spaces is in
some Roumieu space.
\end{abstract}

\maketitle

\section{Introduction}

A general class of hypoelliptic differential operators with variable
coefficients is the class of differential operators of constant strength. H
\"{o}rmander in \cite{H2}, has given a basic result of hypoellipticity of
these operators. Using this result and the theorem of Newberger-Zielezny 
\cite{Z1} concerning the problem of iterates of differential operators with
constant coefficients, Shafii-Mousavi and Zielezny \cite{Z.M} have obtained
a result of Gevrey regularity of solutions of linear differential equations
of constant strength with coefficients in Gevrey classes. In the case of
differential operators with constant coefficients, various results of Gevrey
and Roumieu hypoellipticity have obtained with help of iterates of operators
or more generally of iterates of systems of operators, see in this sense the
works of \cite{BCh01}, \cite{JH} and \cite{Z1}.

Our purpose in this paper is to give an extension of the result of H\"{o}
rmander \cite{H2} in the classes of ultradifferentiable functions (or
classes of Roumieu) to linear differential equations of constant strength
with coefficients in Roumieu spaces.

For the definition of Roumieu spaces we will consider sequences of positive
real numbers $\left( M_{p}\right) $ satisfying the following conditions:
\newline
logarithmic convexity : 
\begin{equation}
M_{0}=1\text{ and }M_{p}^{2}\leq M_{p-1}M_{p+1},\text{ }\forall p\in \mathbb{N}
^{\ast },  \label{H1}
\end{equation}
stability under derivation and multiplication: 
\begin{equation}
\exists H>0:C_{p}^{j}M_{p-j}M_{j}\leq M_{p}\leq H^{p}M_{p-j}M_{j},\text{ }
\forall p\in \mathbb{N},\text{ }j\leq p.  \label{H3}
\end{equation}

\begin{example}
The sequence $M_{p}=p!^{s},$ $s\geq 1,$ called Gevrey sequence of order $s,$
satisfies the conditions $\left( \ref{H1}\right) -\left( \ref{H3}\right) $.
\end{example}

For more details on the spaces of ultra-differentiable functions and ultra-distributions see \cite{Km}.

\begin{definition}
Let $\Omega $ be an open subset of $\mathbb{R}^{n},$ $Q\left( D\right) =\sum_{\left\vert \alpha \right\vert \leq
m}a_{\alpha }D^{\alpha }$ a linear differential operator with constant
coefficients of order $m$. We call Roumieu vector (or vector of type $M_{p}$
) of the operator $Q\left( D\right) $ in $\Omega ,$ any function $u\in
C^{\infty }\left( \Omega \right) $ such that 
\begin{equation*}
\forall H\text{ compact of }\Omega ,\exists C>0,\forall l\in \mathbb{Z}
_{+}:\left\Vert Q^{l}\left( D\right) u\right\Vert _{L^{2}\left( H\right)
}\leq C^{l+1}M_{lm}
\end{equation*}
The space of Roumieu vectors of $Q\left( D\right) $ in $\Omega $ is denoted $
R_{M}\left( \Omega ,Q\right) .$
\end{definition}

\begin{definition}
We call Roumieu space in $\Omega ,$ and we denote $R_{M}\left( \Omega
\right) $, the space of functions $u\in C^{\infty }\left( \Omega \right) $
such that 
\begin{equation*}
\forall H\text{ compact }\Omega ,\exists C>0,\forall \alpha \in \mathbb{Z}
_{+}^{n}:\left\Vert D^{\alpha }u\right\Vert _{L^{2}\left( H\right) }\leq
C^{\left\vert \alpha \right\vert +1}M_{\left\vert \alpha \right\vert }
\end{equation*}
\end{definition}

\begin{example}
If $M_{p}=p!^{s},$ $s\geq 1,$ then $R_{M}\left( \Omega \right) $ is the
Gevrey space of order $s$ in $\Omega ,$ it is denoted $G^{s}\left( \Omega
\right) .$

Similarly $R_{M}\left( \Omega ,Q\right) $ is denoted $G^{s}\left( \Omega
,Q\right) .$
\end{example}

\begin{definition}
A differential operator with variable coefficients $Q\left( x,D\right) $ is
said hypoelliptic if for every open subset $\Omega $ $\subseteq $ $
\mathbb{R}^{n}$ and any distribution $u\in \mathcal{D}^{\prime }\left( \Omega
\right) ,$ we have 
\begin{equation*}
Qu\in C^{\infty }\left( \Omega \right) \Longrightarrow u\in C^{\infty }\left( \Omega \right) .
\end{equation*}
\end{definition}

In the case of differential operators with constant coefficients the
hypoellipticity of these operators can be characterized with help of their
symbols and various characterizations where obtained, see for instance \cite%
{T}. We recall here one of these characterizations.

\begin{proposition}
\label{prop1.1}The operator $Q\left( D\right) $ is hypoelliptic if and only
if 
\begin{equation}
\exists C>0,\exists d\geq 1,\forall \beta \in \mathbb{Z}_{+}^{n},\forall \xi \in 
\mathbb{R}^{n}:\left\vert \xi \right\vert ^{\frac{\left\vert \beta \right\vert }{d}
}\left\vert Q^{\left( \beta \right) }\left( \xi \right) \right\vert \leq
C\left( 1+\left\vert Q\left( \xi \right) \right\vert \right) \text{ },
\label{1.1}
\end{equation}
where $Q^{\left( \beta \right) }\left( \xi \right) =\partial _{\xi }^{\beta
}Q\left( \xi \right) .$
\end{proposition}

\begin{remark}
In fact, from \cite{T} if $Q\left( D\right) $ is hypoelliptic, then there
exists a smallest number $d$, which is rational satisfying $\left( \ref{1.1}
\right) .$ In the sequel we will say that $Q\left( D\right) $ is $d-$
hypoelliptic.
\end{remark}

\section{Preliminaries}

For any open subset $\omega $ of $\mathbb{R}^{n}$ and $\delta >0$ we set 
\begin{equation*}
\omega _{\delta }=\left\{ x\in \omega ,\text{ }d\left( x,C\omega \right)
>\delta \right\} .
\end{equation*}

If $f\in L_{loc}^{2}\left( \omega \right) ,$ $\mu >0$ and $t>0,$ we define 
\begin{equation*}
N_{\omega ,\mu ,t}\left( f\right) =\underset{0<\delta \leq t}{\sup }\delta
^{\mu }\left\Vert f\right\Vert _{L^{2}\left( \omega _{\delta }\right) }.
\end{equation*}

Without loosing of generality, we suppose in which follows that $\omega $ is
a bounded open subset of diameter $<1,$ and for simplify we denote $
N_{\omega ,\mu ,t}\left( f\right) $ by $N_{\mu }\left( f\right) $ if there
is not confusion.

The following proposition is the origin of the fundamental estimate of
hypoelliptic differential operators with constant coefficients, see H\"{o}rmander \cite[Theorem \ 4.2]{H1}.

\begin{proposition}
\label{P.1}Let $\omega $ be a bounded open subset of $\mathbb{R}^{n},$ $Q\left( D\right) $ a $d-$hypoelliptic differential operator of order 
$m,$ and $R\left( D\right) $ a differential operator such that 
\begin{equation*}
\left\vert R\left( \xi \right) \right\vert \leq C\left( 1+\left\vert Q\left(
\xi \right) \right\vert \right) \text{ },\ \forall \xi \in \mathbb{R}^{n}
\end{equation*}
for some $C>0,$ then there exists $C^{\prime }>0$ such that 
\begin{equation*}
\sum_{\alpha }N_{dm}\left( R\left( D\right) u\right) \leq C^{\prime }\left(
N_{dm}\left( Q\left( D\right) u\right) +\left\Vert u\right\Vert
_{L^{2}\left( \omega \right) }\right) ,\text{ }\forall u\in C^{\infty
}\left( \omega \right) .
\end{equation*}
\end{proposition}

From this proposition Newberger and Zielezny in their paper \cite{Z1} have
obtained a result of Gevrey regularity of hypoelliptic differential
operators with constant coefficients. Further the authors in \cite{BCh01}
and \cite{CH.M} have considered a large class of differential operators with
constant coefficients including the class of hypoelliptic differential
operators, and have shown the inclusion between the spaces of Gevrey vectors
in \cite{BCh01} and between the spaces of Roumieu vectors in \cite{CH.M}.

In order to establish the result of Roumieu regularity of solutions of
linear differential equations of constant strength with coefficients in
Roumieu spaces, we enumerate and prove some properties of sequences $\left(
M_{p}\right) $ satisfying the conditions $\left( \ref{H1}\right) -\left( \ref
{H3}\right) .$

\begin{proposition}
If the sequence $\left( M_{p}\right) $ satisfies the condition $\left( \ref
{H1}\right) $, then $\left( M_{p}\right) ^{\frac{1}{p}}$ is an increasing
sequence.
\end{proposition}

\begin{proof}
The growing of $\left( M_{p}\right) ^{\frac{1}{p}}$ is equivalent to 
\begin{equation}
\frac{p+1}{p}\log M_{p}\leq \log M_{p+1}\text{, }\forall p\in \mathbb{N}^{\ast }  \label{4}
\end{equation}
We will prove $\left( \ref{4}\right) $ by induction on $p$. For $p=1$, the
relation $\left( \ref{4}\right) $ is obtained from $\left( \ref{H1}\right) $
by composing the $\log $ function. Suppose that the estimate $\left( \ref{4}
\right) $ takes place until $p$ and verify that it remains true for $p+1.$
From condition $\left( \ref{H1}\right) $ and the hypothesis of recurrence,
we have%
\begin{eqnarray*}
2\log M_{p+1} &\leq &\log M_{p}+\log M_{p+2} \\
&\leq &\frac{p}{p+1}\log M_{p+1}+\log M_{p+2},
\end{eqnarray*}
hence 
\begin{equation*}
2\log M_{p+1}-\frac{p}{p+1}\log M_{p+1}\leq \log M_{p+2},
\end{equation*}
i.e. 
\begin{equation*}
\frac{p+2}{p+1}\log M_{p+1}\leq \log M_{p+2},
\end{equation*}
which is the relation $\left( \ref{4}\right) $ at order $p+1.$
\end{proof}

\begin{proposition}
If the sequence $\left( M_{p}\right) $ satisfies the condition $\left( \ref
{H3}\right) $ we have 
\begin{equation}
\forall m\in \mathbb{Q},\text{ }\exists B\geq 0:M_{pm}\leq B^{p}\left(
M_{p}\right) ^{m},\text{ }\forall p\in \mathbb{N},\ pm\in \mathbb{N}  \label{H4}
\end{equation}
\end{proposition}

\begin{proof}
First let's prove $\left( \ref{H4}\right) $ by induction on $m$ if $m\in 
\mathbb{N}.$ For $m=2$ the estimate is fulfilled from the condition $\left( 
\ref{H3}\right) $ with $B=H^{2}.$ So suppose that $\left( \ref{H4}\right) $
is satisfied until the order $m$ and let's prove it at the order $m+1.$ We
have from the condition $\left( \ref{H3}\right) $ and the recurrence
hypothesis
\begin{eqnarray*}
M_{p\left( m+1\right) } &\leq &H^{p\left( m+1\right) }M_{p}M_{pm} \\
&\leq &H^{p\left( m+1\right) }M_{p}B^{p}\left( M_{p}\right) ^{m} \\
&\leq &\left( H^{\left( m+1\right) }B\right) ^{p}\left( M_{p}\right) ^{m+1},
\end{eqnarray*}
which proves that $\left( \ref{H4}\right) $ is satisfied at the order $m+1.$

If $m\in \mathbb{Q}$ set $m=\frac{r}{s}$ in an irreducible form, so if$\
pm\in \mathbb{N}$ then $p=st$ with $t\in \mathbb{N}.$ Applying the
inequality $\left( \ref{H4}\right) $ for $t$ and $r$ we obtain 
\begin{equation*}
M_{pm}=M_{tr}\leq B_{r}^{t}\left( M_{t}\right) ^{r}
\end{equation*}%
On the other hand the sequence $\left( M_{p}\right) ^{\frac{1}{p}}$ is
increasing, so 
\begin{equation*}
\forall h\leq p,\text{ } M_{h}\leq \left( M_{p}\right) ^{\frac{h}{p}}
\end{equation*}
In particular we get from the last inequality
\begin{equation*}
M_{pm}\leq B_{r}^{t}\left( M_{st}\right) ^{\frac{r}{s}}\leq B_{m}^{p}\left(
M_{p}\right) ^{m}
\end{equation*}
\end{proof}

\begin{definition}
We denote $M_{p}\subset N_{p}$ if 
\begin{equation*}
\exists L>0,\exists C>0:M_{p}\leq CL^{p}N_{p},\ \forall p\in \mathbb{N}
\end{equation*}
\end{definition}

\begin{example}
We have $p!\subset M_{p}$, for all sequence $M_{p}$ satisfying condition $%
\left( \ref{H3}\right) $.
\end{example}

\begin{remark}
In fact we have a more stronger estimate, see \cite{Km}, 
\begin{equation*}
\forall L>0,\exists C>0:p!\leq CL^{p}M_{p},\ \forall p\in \mathbb{N}
\end{equation*}
\end{remark}

\section{Roumieu regularity of hypoelliptic operators with constant
coefficient}

The first result of this paper is the Roumieu regularity of hypoelliptic
operators with constant coefficients. But first let us show the following estimate of higher order of hypoelliptic operators with constant coefficients.

For simplify the notations set $d=\frac{\mu }{\nu }$ and $\gamma =dm\mu $
where $\mu ,\nu \in \mathbb{N}^{\ast }.$

\begin{proposition}
\label{prop3.1}Let $\omega $ be a bounded openk subset of $\mathbb{R}^{n}$ and let $Q$ be a $d-$hypoelliptic operator of order $m,$ then

$\exists C>0,\forall k\geq 0,\forall \alpha \in \mathbb{Z}_{+}^{n},\left\vert \alpha \right\vert \leq km\nu ,\forall \delta >0,\forall
u\in C^{\infty }\left( \omega \right) ,$
\begin{equation}
\left\Vert D^{\alpha }u\right\Vert _{L^{2}\left( \omega _{\delta }\right)
}\leq C^{k}\sum_{i=0}^{k}\binom{k}{i}\left( \frac{k}{\delta }\right)
^{\left( k-i\right) dm}\left\Vert Q^{i}u\right\Vert _{L^{2}\left( \omega
\right) }  \label{3.1}
\end{equation}
\end{proposition}

\begin{proof}
Since $Q$ is $d-$hypoelliptic, so choosing $\beta $\ in $\left( \ref{1.1}
\right) $\ such that $\left\vert \beta \right\vert =m$ and $Q^{\left( \beta
\right) }\left( D\right) =Cst\neq 0$ we get
\begin{equation*}
\left\vert \xi \right\vert ^{m\nu }\leq C\left( 1+\left\vert Q^{\mu }\left(
\xi \right) \right\vert \right) ,\ \forall \xi \in \mathbb{R}^{n}
\end{equation*}
for some $C>0,$ then for all $\alpha \in \mathbb{Z}_{+}^{n}$ with $\left\vert \alpha \right\vert \leq m\nu $ we have 
\begin{equation*}
\left\vert \xi ^{\alpha }\right\vert \leq \left\vert \xi \right\vert
^{\left\vert \alpha \right\vert }\leq C\left( 1+\left\vert Q^{\mu }\left(
\xi \right) \right\vert \right) ,\ \forall \xi \in \mathbb{R}^{n}
\end{equation*}
On the other hand we can check from Proposition \ref{prop1.1}\ that the
operator $Q^{\mu }\left( D\right) $ which is of order $m\mu $\ is also $d-$
hypoelliptic, hence from proposition \ref{P.1}, for any bounded open subset $
\omega $ of $\mathbb{R}^{n}$ there exists $C_{1}>0$ such that for all $\alpha \in \mathbb{Z}
_{+}^{n}$ with $\left\vert \alpha \right\vert \leq m\nu $ we have%
\begin{equation*}
N_{\gamma }\left( D^{\alpha }u\right) \leq C_{1}\left( N_{\gamma }\left(
Q^{\mu }u\right) +\left\Vert u\right\Vert _{L^{2}\left( \omega \right)
}\right) ,\text{ }\forall u\in C^{\infty }\left( \omega \right)
\end{equation*}
By definition of $N_{\gamma },$ we obtain

$\forall t>0,$ $\forall \rho \geq 0,$ $\forall \upsilon \in C^{\infty
}\left( \omega \right) ,$ $\forall \alpha \in \mathbb{Z}_{+}^{n},\ \left\vert \alpha \right\vert \leq m\nu :$
\begin{equation*}
\underset{0<\tau \leq t}{\sup \tau ^{\gamma }}\left\Vert D^{\alpha }\upsilon
\right\Vert _{L^{2}\left( \omega _{\rho +\tau }\right) }\leq C_{1}\left( 
\underset{0<\tau \leq t}{\sup \tau ^{\gamma }}\left\Vert Q^{\mu }\upsilon
\right\Vert _{L^{2}\left( \omega _{\rho +\tau }\right) }+\left\Vert \upsilon
\right\Vert _{L^{2}\left( \omega _{\rho }\right) }\right) ,
\end{equation*}
therefore
\begin{equation}
\left\Vert D^{\alpha }\upsilon \right\Vert _{L^{2}\left( \omega _{\rho
+t}\right) }\leq C_{1}\left( \left\Vert Q^{\mu }\upsilon \right\Vert
_{L^{2}\left( \omega _{\rho }\right) }+t^{-\gamma }\left\Vert \upsilon
\right\Vert _{L^{2}\left( \omega _{\rho }\right) }\right) ,\text{ }\forall
\upsilon \in C^{\infty }\left( \omega _{\rho }\right) .  \label{3.2}
\end{equation}
From this estimate we will show $\left( \ref{3.1}\right) $ by recurrence on $
k.$ For $k=0$ it is trivial and for $k=1$\ the estimate $\left( \ref{3.1}
\right) $ is obtained from $\left( \ref{3.2}\right) $ taking $\rho =0$\ and $
t=\delta .$ Suppose that $\left( \ref{3.1}\right) $\ takes place until the
order $k$ and let's prove that it remains true at the order $k+1.$ Let $
\alpha \in \mathbb{Z}_{+}^{n}$ with $km\nu <\left\vert \alpha \right\vert \leq \left( k+1\right)
m\nu ,$ writing $\alpha =\alpha ^{\prime }+\alpha _{0}$ with $\left\vert
\alpha _{0}\right\vert =m\nu $ and substituting in $\left( \ref{3.2}\right)
, $ for $t=\frac{\delta }{k+1}$, $\rho =\frac{k\delta }{k+1}$ and $\upsilon
=D^{\alpha ^{\prime }}u$ we obtain 
\begin{eqnarray*}
\left\Vert D^{\alpha }u\right\Vert _{L^{2}\left( \omega _{\delta }\right) }
&\leq &C_{1}\left( \left\Vert Q^{\mu }\left( D^{\alpha ^{\prime }}u\right)
\right\Vert _{\omega _{\rho }}+\left( \frac{k+1}{\delta }\right) ^{\gamma
}\left\Vert D^{\alpha ^{\prime }}u\right\Vert _{L^{2}\left( \omega _{\rho
}\right) }\right) \\
&\leq &C_{1}C_{1}^{k}\sum_{i=0}^{k}\binom{k}{i}\left( \frac{k+1}{\delta }
\right) ^{\left( k-i\right) \gamma }\left\Vert Q^{\mu i}Q^{\mu }u\right\Vert
_{L^{2}\left( \omega \right) } \\
&&+C_{1}C_{1}^{k}\sum_{i=0}^{k}\binom{k}{i}\left( \frac{k+1}{\delta }\right)
^{\left( k-i\right) \gamma }\left\Vert Q^{\mu i}u\right\Vert _{L^{2}\left(
\omega \right) } \\
&\leq &C_{1}^{k+1}\sum_{i=1}^{k+1}\binom{k}{i-1}\left( \frac{k+1}{\delta }
\right) ^{\left( k-i+1\right) \gamma }\left\Vert Q^{\mu i}u\right\Vert
_{L^{2}\left( \omega \right) } \\
&&+C_{1}^{k+1}\sum_{i=0}^{k}\binom{k}{i}\left( \frac{k+1}{\delta }\right)
^{\left( k-i\right) \gamma }\left\Vert Q^{\mu i}u\right\Vert _{L^{2}\left(
\omega \right) } \\
&\leq &C_{1}^{k+1}\sum_{i=0}^{k+1}\binom{k+1}{i}\left( \frac{k+1}{\delta }
\right) ^{\left( k+1-i\right) \gamma }\left\Vert Q^{\mu i}u\right\Vert
_{L^{2}\left( \omega \right) }.
\end{eqnarray*}
\end{proof}

\begin{theorem}
\label{Th.1}Let $\Omega $ be an open subset of $\mathbb{R}^{n}$ and $Q$ a $
d- $hypoelliptic differential operator with constant coefficients, and let $
\left( M_{p}\right) $ be a sequence satisfying $\left( \ref{H1}\right)
-\left( \ref{H3}\right) ,$ and such that 
\begin{equation}
\left( p!\right) ^{d}\subset M_{p},  \label{3.10}
\end{equation}
then 
\begin{equation*}
R_{M}\left( \Omega ,Q\right) \subset R_{M^{d}}\left( \Omega \right) .
\end{equation*}
\end{theorem}

\begin{proof}
Suppose that $u\in R_{M}\left( \Omega ,P\right) $, and let $H$ be a compact
subset of $\Omega ,$ then there exists a bounded open set $\omega $ and $
\delta >0$ such that $H\subset \omega _{\delta }\subset \omega \subset
\Omega .$ Therefore there exists $A>0$ such that 
\begin{equation*}
\left\Vert Q^{\mu i}u\right\Vert _{L^{2}\left( \omega \right) }\leq A^{\mu
i+1}M_{\mu im}\text{ },\text{ }i=0,1,...,
\end{equation*}
hence from $\left( \ref{3.10}\right) $ and since $p^{p}\leq p!e^{p},$ we get
for every $i\leq k+1,$
\begin{eqnarray}
k^{\left( k+1-i\right) \gamma }\left\Vert Q^{\mu i}u\right\Vert
_{L^{2}\left( \omega \right) } &\leq &\left( \left( k+1\right) m\mu \right)
!^{\frac{k+1-i}{k+1}d}e^{\left( k+1\right) m\mu }A^{\mu i+1}M_{i\mu m} 
\notag \\
&\leq &A_{1}^{k+1}M_{im\mu }\left( M_{\left( k+1\right) m\mu }\right) ^{
\frac{k+1-i}{k+1}},  \label{3.11}
\end{eqnarray}
for some constant $A_{1}>0.$ But the sequence $\left( M_{p}\right) ^{\frac{1}{p}}$ is increasing, so
\begin{equation*}
\forall h\leq p,\text{ }\left( M_{h}\right) ^{p}\leq \left( M_{p}\right) ^{h}
\end{equation*}
Applying for $h=im\mu $ and $p=\left( k+1\right) m\mu ,$ we obtain 
\begin{equation*}
M_{im\mu }\leq \left( M_{\left( k+1\right) m\mu }\right) ^{\frac{im\mu }{
\left( k+1\right) m\mu }}=\left( M_{\left( k+1\right) m\mu }\right) ^{\frac{i
}{\left( k+1\right) }},
\end{equation*}
which gives with $\left( \ref{3.11}\right) $ and $\left( \ref{H3}\right) $
\begin{eqnarray*}
k^{\left( k+1-i\right) \gamma }\left\Vert Q^{\mu i}u\right\Vert
_{L^{2}\left( \omega \right) } &\leq &A_{1}^{k+1}M_{\left( k+1\right) m\mu }
\\
&\leq &A_{1}^{k+1}H^{\left( k+1\right) m\mu }M_{km\mu }M_{m\mu } \\
&\leq &A_{2}^{k+1}M_{km\mu }.
\end{eqnarray*}
This estimate will permit us to conclude. Let $\alpha \in \mathbb{Z}
_{+}^{n}, $ so there exists an integer $k$ such that $\left( k-1\right) \nu
m\leq \left\vert \alpha \right\vert \leq k\nu m.$ Taking account of the
condition $\left( \ref{H3}\right) $, we obtain from $\left( \ref{3.1}\right) $ and the last estimate 
\begin{eqnarray*}
\left\Vert D^{\alpha }u\right\Vert _{L^{2}\left( H\right) } &\leq
&C^{k}\sum_{i=0}^{k+1}\binom{k}{i}\left( \frac{1}{\delta }\right) ^{\left(
k-i\right) \gamma }A_{2}^{k+1}M_{km\mu } \\
&\leq &C^{k}\left( \left( \frac{1}{\delta }\right) ^{\gamma }+1\right)
^{k}A_{2}^{k+1}M_{km\mu } \\
&\leq &A_{3}^{k+1}M_{km\mu }
\end{eqnarray*}
But $k\mu m=k\nu dm\leq \left\vert \alpha \right\vert d,$ so we get with the
inequality $\left( \ref{H4}\right) $ 
\begin{equation*}
\left\Vert D^{\alpha }u\right\Vert _{L^{2}\left( H\right) }\leq
A_{3}^{\left\vert \alpha \right\vert +1}M_{\left\vert \alpha \right\vert
}^{d},
\end{equation*}
which shows that $u\in R_{M^{d}}\left( \Omega \right) .$
\end{proof}

\section{Roumieu regularity of hypoelliptic operators with constant strength}

Before establish the regularity of solutions of differential equations
associated with differential operators of constant strength, we will recall
some definitions and properties related to this question and the fundamental
result due to H\"{o}rmander \cite{H2} of regularity of these operators.

\begin{definition}
A differential operator $P\left( x,D\right) $ with variable coefficients in $
\Omega $ is said of constant strength in $\Omega ,$ if for arbitrary fixed $
x,y\in \Omega $, the differential operators $P\left( x,D\right) $ and $
P\left( y,D\right) $ with constant coefficients are equally strong, that
means 
\begin{equation*}
\forall x,y\in \Omega ,\text{ }\exists C_{x,y}>0:\widetilde{P}\left( x,\xi
\right) \leq C_{x,y}\widetilde{P}\left( y,\xi \right) ,\text{ }\forall \xi
\in \mathbb{R}^{n},
\end{equation*}
where $\widetilde{P}^{2}\left( \xi \right) =\sum_{\alpha }\left\vert
P^{\left( \alpha \right) }\left( \xi \right) \right\vert ^{2}.$
\end{definition}

\begin{definition}
A positive function $h$ defined in $\mathbb{R}^{n}$ will be called a temperate weight function if 
\begin{equation}
\exists C,\exists N>0:h\left( \xi +\eta \right) \leq \left( 1+C\left\vert
\eta \right\vert \right) ^{N}h\left( \xi \right) ;\text{ }\forall \xi ,\eta
\in \mathbb{R}^{n}.  \label{3.13}
\end{equation}
The set of these functions will be denoted $\mathcal{H}.$
\end{definition}

\begin{example}
For every polynomial $P\left( \xi \right) $ the function $\widetilde{P}
\left( \xi \right) $ is in $\mathcal{H}$.
\end{example}

\begin{lemma}
\label{Lemma1}For all $h\in \mathcal{H}$ and all $\delta >0,$ the function 
\begin{equation*}
h_{\delta }\left( \xi \right) =\sup_{\left\vert \eta \right\vert \leq \delta
}h\left( \xi +\eta \right) ,
\end{equation*}
is a function in $\mathcal{H}$ satisfying
\begin{equation}
\exists C>0,\exists N>0:h\left( \xi \right) \leq h_{\delta }\left( \xi
\right) \leq h\left( \xi \right) \left( 1+C\delta \right) ^{N},\ \forall \xi
\in \mathbb{R}^{n}, \label{3.15}
\end{equation}
\begin{equation}
\left( h^{j}\right) _{\delta }=\left( h_{\delta }\right) ^{j}\text{ for\ }
j=1,2,...  \label{3.16}
\end{equation}
\end{lemma}

\begin{proof}
First for all $\xi ,\theta \in \mathbb{R}^{n}$ we have
\begin{eqnarray*}
h_{\delta }\left( \xi +\theta \right) &=&\sup_{\left\vert \eta \right\vert
\leq \delta }h\left( \xi +\theta +\eta \right) \leq \sup_{\left\vert \eta
\right\vert \leq \delta }h\left( \xi +\eta \right) \left( 1+C\left\vert
\theta \right\vert \right) ^{N} \\
&\leq &\left( 1+C\left\vert \theta \right\vert \right) ^{N}h_{\delta }\left(
\xi \right) ,
\end{eqnarray*}
so $h_{\delta }\in \mathcal{H}$. Next the first inequality in $\left( \ref
{3.15}\right) $ follows from the definition of $h_{\delta }.$ For the second
inequality, since $h$ is a temperate weight function then there
exist $C>0$ and $N>0$ such that
\begin{equation*}
h_{\delta }\left( \xi \right) \leq \sup_{\left\vert \eta \right\vert \leq
\delta }\left( 1+C\left\vert \eta \right\vert \right) ^{N}h\left( \xi
\right) \leq h\left( \xi \right) \left( 1+C\delta \right) ^{N},\text{ }
\forall \xi \in \mathbb{R}^{n}.
\end{equation*}
Finally the property $\left( \ref{3.16}\right) $ is obtained by recurrence
on $j$.
\end{proof}

For every $h\in \mathcal{H}$ we set 
\begin{equation*}
\mathcal{B}_{p,h}=\left\{ u\in \mathcal{S}^{\prime },\left( \int
\left\vert h\left( \xi \right) \hat{u}\left( \xi \right) \right\vert
^{p}d\xi \right) ^{\frac{1}{p}}<\infty ,\text{ }1\leq p<\infty \right\} ,
\end{equation*}
provided with the norm
\begin{equation*}
\left\Vert u\right\Vert _{p,h}=\left( \int \left\vert h\left( \xi \right) 
\hat{u}\left( \xi \right) \right\vert ^{p}d\xi \right) ^{\frac{1}{p}}.
\end{equation*}
If $p=\infty ,$ then $\left\Vert u\right\Vert _{\infty ,h}=ess\sup
\left\vert h\left( \xi \right) \hat{u}\left( \xi \right) \right\vert .$

The following theorem is the basic result of hypoellipticity of differential
operators of constant strength, see H\"{o}rmander \cite[Theorem 7.3.1 and
Theorem 7.4.1]{H2}.

\begin{theorem}
\label{Th. Horm}Let $P\left( x,D\right) $ be a differential operator with $
C^{\infty }$ coefficients and of constant strength in a neighborhood of $x_{0}\in 
\mathbb{R}^{n}.$ Then there exists a sufficiently small open neighborhood $\Omega $ of 
$x_{0},$ and a linear mapping $E$ of $\mathcal{E}^{^{\prime }}\left( \mathbb{R}
^{n}\right) $ into $\mathcal{E}^{^{\prime }}\left( \mathbb{R}
^{n}\right) $ with the following properties:
\begin{equation*}
P\left( x,D\right) Eu=u\text{ in }\Omega \text{ if }u\in \mathcal{E}
^{^{\prime }}\left( \mathbb{R}^{n}\right) ,
\end{equation*}
\begin{equation}
EP\left( x,D\right) \upsilon =\upsilon \text{ in }\Omega \text{ if\ }
\upsilon \in \mathcal{E}^{^{\prime }}\left( \Omega \right) ,  \label{3.17}
\end{equation}
\begin{equation}
\left\Vert Eu\right\Vert _{p,\widetilde{P}_{0}h}\leq C_{h}\left\Vert
u\right\Vert _{p,h}\text{ if\ }u\in \mathcal{E}^{^{\prime }}\left( \Omega
\right) \cap \mathcal{B}_{p,h},\text{ }h\in \mathcal{H},  \label{3.18}
\end{equation}
where $P_{0}\left( D\right) =P\left( x_{0},D\right) $ and $C_{h}$ is
independent of $u.$

If in addition $P_{0}\left( D\right) $ is hypoelliptic then $P\left(
x,D\right) $ is so also.
\end{theorem}

\begin{remark}
In the proof of this result, H\"{o}rmander has shown the following estimate
which is stronger than $\left( \ref{3.18}\right) ,$ see \cite[p. 176]{H2}.
\begin{equation}
\left\Vert Eu\right\Vert _{p,\widetilde{P}_{0}h_{\delta }}\leq C\left\Vert
u\right\Vert _{p,h_{\delta }}\text{ if\ }u\in \mathcal{E}^{^{\prime }}\left(
\Omega \right) \cap \mathcal{B}_{p,h},\text{ }h\in \mathcal{H},  \label{4.1}
\end{equation}
if $\delta $ is sufficiently small. The constant $C$\ is independent of $u,h$
and $\delta .$
\end{remark}

We will use this estimate in the proof of our second result of this paper
which is the following theorem.

\begin{theorem}
\label{Th2}Let $P\left( x,D\right) $ be a differential operator of constant
strength in a neighborhood of $x_{0},$ and let $\left( M_{p}\right) $ a
sequence satisfying the conditions $\left( \ref{H1}\right) -\left( \ref{H3}
\right) .$ Then there exists a open neighborhood $\Omega _{0}$ of $
x_{0}$\ such that, if the coefficients of $P\left( x,D\right) $ are in $
R_{M}\left( \Omega _{0}\right) $ and the operator $P_{0}\left( D\right)
=P\left( x_{0},D\right) $ is $d-$hypoelliptic and the sequence $\left(
M_{p}\right) $ satisfyies the condition $\left( \ref{3.10}\right) ,$ we have
\begin{equation*}
\forall u\in \mathcal{D}^{\prime }\left( \Omega _{0}\right) ,\ P\left(
x,D\right) u\in R_{M}\left( \Omega _{0}\right) \Rightarrow u\in
R_{M^{d}}\left( \Omega _{0}\right) .
\end{equation*}
\end{theorem}

\begin{proof}
Let $\Omega $ be as in Theorem \ref{Th. Horm} and let us show that $
\left( \ref{4.1}\right) $ gives the following estimate
\begin{equation}\label{4.2}
\exists A>0,\ \left\Vert E^{l}u\right\Vert _{2,\left( \tilde{P}_{0}\right)
^{l}}\leq A^{l}\left\Vert u\right\Vert _{L^{2}},\forall u\in C_{0}^{\infty
}\left( \Omega \right) ,\ l=0,1,...,
\end{equation}
Indeed applying for $p=2,$ $h^{l-1}$\
instead of $h$\ and $E^{l-1}u$\ instead of $u$ we obtain 
\begin{equation}
\left\Vert E^{l}u\right\Vert _{2,\widetilde{P}_{0}(h^{l-1})_{\delta
}}=\left\Vert EE^{l-1}u\right\Vert _{2,\widetilde{P}_{0}(h^{l-1})_{\delta
}}\leq C\left\Vert E^{l-1}u\right\Vert _{2,(h^{l-1})_{\delta }},  \label{4.3}
\end{equation}
but from the properties $\left( \ref{3.15}\right) $ and $\left( \ref{3.16}
\right) $ of $h_{\delta }$\ we have 
\begin{eqnarray*}
(h^{l-1})_{\delta }\left( \xi \right)  &=&\left( h_{\delta }\left( \xi
\right) \right) ^{l-1}=h_{\delta }\left( \xi \right) \left( h_{\delta
}\left( \xi \right) \right) ^{l-2}=h_{\delta }\left( \xi \right) \left(
h^{l-2}\right) _{\delta }\left( \xi \right)  \\
&\leq &h\left( \xi \right) \left( 1+C\delta \right) ^{N}(h^{l-2})_{\delta
}\left( \xi \right)  \\
&\leq &C_{1}h\left( \xi \right) (h^{l-2})_{\delta }\left( \xi \right) 
\end{eqnarray*}
Substituting in $\left( \ref{4.3}\right) $ and replacing $h$ with $
\widetilde{P}_{0}$ we get
\begin{equation*}
\left\Vert E^{l}u\right\Vert _{2,\widetilde{P}_{0}(\widetilde{P}
_{0}^{l-1})_{\delta }}\leq CC_{1}\left\Vert E^{l-1}u\right\Vert _{2,
\widetilde{P}_{0}(\widetilde{P}_{0}^{l-2})_{\delta }}
\end{equation*}
Repeating the same argument we get at the $l-th$ step
\begin{equation*}
\left\Vert E^{l}u\right\Vert _{2,\widetilde{P}_{0}(\widetilde{P}
_{0}^{l-1})_{\delta }}\leq \left( CC_{1}\right) ^{l}\left\Vert u\right\Vert
_{L^{2}},
\end{equation*}
which implies from the property $\left( \ref{3.15}\right) $ 
\begin{equation*}
\left\Vert E^{l}u\right\Vert _{2,\tilde{P}_{0}^{l}}\leq \left\Vert
E^{l}u\right\Vert _{2,\widetilde{P}_{0}(\widetilde{P}_{0}^{l-1})_{\delta
}}\leq \left( CC_{1}\right) ^{l}\left\Vert u\right\Vert _{L^{2}}
\end{equation*}
Thus the estimate $\left( \ref{4.2}\right) $ is fulfilled with $A=CC_{1}.$

Applying $\left( \ref{4.2}\right) $ for $P^{l}\left( x,D\right) u$ instead
of $u$ we obtain 
\begin{equation*}
\left\Vert E^{l}P^{l}\left( x,D\right) u\right\Vert _{2,\tilde{P}
_{0}^{l}}\leq A^{l}\left\Vert P^{l}\left( x,D\right) u\right\Vert _{L^{2}},\
\forall u\in C_{0}^{\infty }\left( \Omega \right) ,\ l=0,1,...,
\end{equation*}
On the other hand 
\begin{equation*}
\left\Vert P_{0}^{l}\left( D\right) u\right\Vert _{L^{2}}\leq \left\Vert
u\right\Vert _{2,\tilde{P}_{0}^{l}},\text{ }\forall u\in C_{0}^{\infty
}\left( \Omega \right) ,
\end{equation*}
then for all $u\in C_{0}^{\infty }\left( \Omega \right) $ and $l=0,1,...,$
we have 
\begin{equation*}
\left\Vert P_{0}^{l}\left( D\right) E^{l}P^{l}\left( x,D\right) u\right\Vert
_{L^{2}}\leq \left\Vert E^{l}P^{l}\left( x,D\right) u\right\Vert _{2,\tilde{P
}_{0}^{l}}\leq A^{l}\left\Vert P^{l}\left( x,D\right) u\right\Vert .
\end{equation*}
But from Theorem \ref{Th. Horm} $EP\left( x,D\right) u=u,\ \forall u\in
C_{0}^{\infty }\left( \Omega \right) ,$ so $E^{l}P^{l}\left( x,D\right) u=u$
for all $l=0,1,...,$ and therefore from the above estimate we get 
\begin{equation}
\left\Vert P_{0}^{l}\left( D\right) u\right\Vert _{L^{2}}\leq
A^{l}\left\Vert P^{l}\left( x,D\right) u\right\Vert  \label{4.4}
\end{equation}

Now let $u\in \mathcal{E}^{^{\prime }}\left( \Omega \right) $ such that $
P\left( x,D\right) u\in R_{M}\left( \Omega \right) .$ By hypothesis the
operator $P\left( x,D\right) $ is of constant strength and the operator $
P_{0}\left( D\right) $ is hypoelliptic, then $P\left( x,D\right) $ is also
hypoelliptic, therefore $u\in C^{\infty }\left( \Omega \right) $ because $
R_{M}\left( \Omega \right) \subset C^{\infty }\left( \Omega \right) .$
Further since the coefficients of $P\left( x,D\right) $ are in $R_{M}\left(
\Omega \right) $ it is easy to check that $u$ is a Roumieu vector of the
operator $P\left( x,D\right) $ in $\Omega ,$ hence $u$ is also a Roumieu
vector of the operator $P_{0}\left( D\right) $ in $\Omega $ according to $
\left( \ref{4.4}\right) .$ From Theorem \ref{Th.1} we conclude that $u\in R_{M^{d}}\left(
\Omega \right) .$
\end{proof}

\begin{corollary}
Let $\Omega $ be an open subset of $\mathbb{R}^{n}$, $P\left(
x,D\right) $ a differential operator of constant strength in $\Omega $
with coefficients in $R_{M}\left( \Omega \right) ,$ and let $\left(
M_{p}\right) $ be a sequence satisfying the conditions $\left( \ref{H1}\right)
-\left( \ref{H3}\right) $. If for some $x_{0}\in \Omega $\ the operator $
P_{0}\left( D\right) =P\left( x_{0},D\right) $ is $d-$hypoelliptic and the
sequence $\left( M_{p}\right) $ satisfyies the condition $\left( \ref{3.10}
\right) ,$ then 
\begin{equation*}
\forall u\in \mathcal{D}^{\prime }\left( \Omega \right) ,\ P\left(
x,D\right) u\in R_{M}\left( \Omega \right) \Rightarrow u\in R_{M^{d}}\left(
\Omega \right) .
\end{equation*}
\end{corollary}

\begin{proof}
First note that since the differential operator $P\left( x,D\right) $ is of
constant strength in $\Omega $ and $P\left( x_{0},D\right) $ is $d-$%
hypoelliptic so $P\left( x_{1},D\right) $ is also $d-$hypoelliptic for all
point $x_{1}\in \Omega ,$ see \cite{T}. Next each compact set $H$ of $\Omega 
$ can be recovered by a finite number of open subset $\Omega _{j}$ as in
Theorem \ref{Th2}.
\end{proof}

\section*{author's contribution}
All authors contributed to the study conception and design.
\section*{conflict of interest}
There is no conflict of interest.

\end{document}